\documentclass[a4paper,UKenglish,pdfa]{lipics-v2021}

\newcommand{\sign}{\textnormal{sign}}
\newcommand{\boldsign}{\textnormal{\textbf{sign}}}
\newcommand{\pf}{\textnormal{pf}}
\newcommand{\spf}{\textnormal{spf}}
\newcommand{\Pf}{\textnormal{Pf}}
\newcommand{\fw}{\textnormal{fw}}
\newcommand{\rev}{\mathbin{\textnormal{rev}}}
\newcommand{\symdiff}{\mathbin{\triangle}}

\newcommand{\rank}{\textnormal{rank}}
\newcommand{\boldalpha}{\boldsymbol{\alpha}}
\newcommand{\boldtau}{\boldsymbol{\tau}}
\newcommand{\gf}{\textnormal{GF}}

\newcommand{\sharpp}{{\footnotesize\#}\textsc{p}}

\pdfoutput=1 
\hideLIPIcs  

\title{Lower Bounds for the Pfaffian Number of Graphs}

\author{Enrique {Junchaya}}{Universidade de São Paulo, Brazil}{enriquejh@usp.br}{}{funded by São Paulo Research Foundation (FAPESP), Brazil. Process number \#2024/02636-4}
\author{Cláudio {Leonardo Lucchesi}}{Universidade de São Paulo, Brazil}{}{}{}
\author{Alberto Alexandre {Assis Miranda}}{Instituto Federal do Norte de Minas Gerais, Montes Claros, Brazil}{}{}{}

\authorrunning{E. Junchaya, C.\,L. Lucchesi, and A.\,A.\,A. Miranda}

\Copyright{Enrique Junchaya, Cláudio\,L. Lucchesi, and Alberto\,A.\,A. Miranda}



\relatedversion{}

\nolinenumbers 

\begin{document}

\maketitle

\begin{abstract}
The number of perfect matchings  of a \emph{$k$-pfaffian graph} can be
counted  by computing  a linear  combination of  the pfaffians  of $k$
matrices.  The \emph{pfaffian  number} of a graph $G$  is the smallest
integer~$k$ such that~$G$ is $k$-pfaffian.  We present the first known
lower bounds  for the pfaffian  number of graphs.  As  an intermediate
step, we prove an upper bound for  the rank of two matrices related to
their Khatri-Rao  product, a  result of independent  relevance. 
One of the consequences of these results is the existence of graphs whose pfaffian numbers are arbitrarily large.
\end{abstract}

\section{Introduction}
\label{sec:introduction}

Counting perfect  matchings of a  graph is a fundamental  problem with
practical    applications.     Its    main   applications    are    in
chemistry~\cite{gutman1977,    herndon1974},    statistical    physics~\cite{brush1967,     kast1967}     and     quantum     mechanics~\cite{kgz2017}.
Examples of some of  such  applications  can  be  found  in  the  book  by  Lovász  and
Plummer~\cite[Chapter 8]{lopl86}.  On the  other hand, this problem is
\sharpp-complete~\cite{vali79},  which  suggests  that  there  are  no
efficient algorithms  to solve it  in general.   In that sense,  it is
interesting to study graph classes  for which efficient algorithms for
coun\-ting perfect matchings are known.   One of the most studied such
classes is the class of pfaffian graphs.

Unless otherwise stated, assume that the graphs in this work are simple.
Let $G$ be a graph and let $D$ be an orientation of $G$.
Let $(1,2,\dots,2k)$ be an enumeration of the vertices in $G$.
Let $M$ be a perfect matching of $G$.
The \emph{permutation} of~$M$ in~$D$, denoted by~$\pi_D(M)$, is
\[
\begin{pmatrix}
	1   & 2   & 3   & 4   & \dots & 2k - 1 & 2k \\
	u_1 & v_1 & u_2 & v_2 & \dots & u_k    & v_k
\end{pmatrix},
\]
where,  for each  $i$, $1  \le i  \le k$,  $u_iv_i$ is  an arc  of $M$
oriented from $u_i$  to $v_i$ in $D$.  The \emph{sign}  of $M$ in~$D$,
denoted by  $\sign_D(M)$, is the  sign of the  permutation $\pi_D(M)$.
Notice  that $\sign_D(M)$  is independent  of the  order in  which the
edges  of $M$  are listed  in  $\pi_D(M)$.  The  orientation~$D$ is  a
\emph{pfaffian orientation} if  and only if all  the perfect matchings
of $G$ have the same sign in  $D$.  A \emph{pfaffian graph} is a graph
that admits a pfaffian orientation.

Denote  the
skew-symmetric  adjacency  matrix of  $D$  by  $\mathbf{A_D} := (a_{ij})$,
defined as follows:
	\[
		a_{ij} = \left\{ 
		\begin{array}{rl}
			1,  & \text{if $ij$ is an arc of } D, \\
			-1, & \text{if $ji$ is an arc of } D, \\
			0,  & \text{otherwise}. \\
		\end{array}
		\right.
	\]
The  \emph{pfaffian} of a
skew-symmetric matrix $\mathbf{A}$  is a polynomial in  the entries of
the   matrix   and   is   denoted  by   $\Pf(\mathbf{A})$.    If   the
orientation~$D$ is  pfaffian, then  $|\Pf(\mathbf{A_D})|$ is  equal to
the number  of perfect matchings of  $G$. It turns out  that, for every
skew-symmetric  matrix, the square of its  pfaffian is  equal  to its
determinant.   Hence,  if $D$  is  pfaffian  then  it is  possible  to
efficiently compute the number of perfect matchings of $G$.

The reader  can find  a detailed  account of the  main results  of the
theory of pfaffian graphs in the book by Lucchesi and Murty~\cite[Part
  III]{lm2024}.

Of course, not  every graph is pfaffian.  The graph  $K_{3, 3}$ is the
smallest   non-pfaffian   graph.    Hence,   we   might   consider   a
generalization of pfaffian graphs.  Let  $k$ be a positive integer.  A
\emph{$k$-orientation} of  $G$ is  a vector  $\mathbf{D} =  (D_1, D_2,
\dots, D_k)$ of $k$ orientations of $G$.  The \emph{sign} of a perfect
matching $M$ in $\mathbf{D}$ is the vector
\[\boldsign_\mathbf{D}(M) := (\sign_{D_1}(M), \sign_{D_2}(M), \dots, \sign_{D_k}(M)).\]
Denote by $\mathcal{M}(G)$  the set of perfect matchings  of $G$.  The
\emph{signature      matrix}      of     $\mathbf{D}$      is      the
matrix~$\boldsign_{\mathbf{D}}$ of  $|\mathcal{M}(G)|$ rows  such that
every one of  them is the sign in $\mathbf{D}$  of a different perfect
matching of $G$.  The  $k$-orientation $\mathbf{D}$ is \emph{pfaffian}
if the system
\begin{equation}\label{eq:signature-matrix-system}
	\boldsign_{\mathbf{D}} \mathbf{x} = \mathbf{1}
\end{equation}
is solvable.
A vector that satisfies~\eqref{eq:signature-matrix-system} is a \emph{solution} of $\mathbf{D}$.
A graph that admits a pfaffian $k$-orientation is \emph{$k$-pfaffian}.
Notice that relabeling the vertices of $G$ either changes the sign of all of its perfect matchings in $D$ or does not change the sign of any of its perfect matchings in~$D$.
Therefore, the property of a graph being $k$-pfaffian is independent of the enumeration of its vertices.
It is known, as we will see later, that every graph admits a $k$-pfaffian orientation for some integer $k$.
The \emph{pfaffian number} of $G$, denoted by $\pf(G)$, is the smallest integer $k$ such that $G$ is $k$-pfaffian.
In particular, the pfaffian number of a pfaffian graph is~1.

Given a $k$-pfaffian orientation $\mathbf{D}$ of $G$ with solution $\boldalpha$, the number of perfect matchings of $G$ can be computed by
	\[\sum_{i = 1}^k \alpha_i \Pf(\mathbf{A_{D_i}}).\]

Kasteleyn~\cite{kast61}, in 1961, introduced pfaffian orientations for solving problems in statistical physics.
Later, in 1963, he proved that every planar graph is pfaffian~\cite{kast63}.
In his paper of 1963, he also claimed, without proof, that is possible to count the number of perfect matchings of a graph embeddable in a surface of genus $g$ using $4^g$ pfaffians.
This claim was formally proved in the end of the 20th century by Gallucio and Loebl~\cite{galo99} and, independently, by Tesler~\cite{tesl00}.
\begin{theorem}\label{thm:pf-genus}
	Every graph embeddable in an orientable surface of genus $g$ is $4^g$-pfaffian.
	\qed
\end{theorem}
Tesler also extended this result for non-orientable surfaces.
In 2009, Norine~\cite{nori09} proved that not every positive integer is the pfaffian number of a graph:
\begin{theorem}[Norine,~\cite{nori09}]\label{thm:norine-pfaffian-numbers}
	Every $3$-pfaffian graph is pfaffian and every $5$-pfaffian graph is~$4$-pfaffian.
	\qed
\end{theorem}
In particular,  the pfaffian  number of  $K_{3,3}$, embeddable  in the
torus, is four.   Due to this and other results,  Norine conjectured that
the pfaffian  numbers are always a  power of four.  However,  in 2011,
Miranda  and Lucchesi~\cite{milu11}  showed a  counterexample to  that
conjecture: a  graph with pfaffian number  six.  
Norine~\cite{nori09} also characterized  graphs with  pfaffian number  four.  More
recently,  in   2021,  Costa  Moço,  Miranda   and  Nunes da Silva~\cite{cms21}
characterized   graphs   with   pfaffian  number   six.    These   two
characterizations  are similar  but it  is not  known if  they can  be
generalized.

\subsection{Preliminaries and Notation}

Let $G$ be a graph.  The graph  $G$ is \emph{matchable} if it admits a
perfect  matching.   An edge  of  $G$  is  \emph{matchable} if  it  is
contained  in  some   perfect  matching  of  $G$.   A   graph  $G$  is
\emph{matching covered} if it is connected,  has at least one edge and
all of its edges are matchable.   The study of $k$-pfaffian graphs can
be naturally  reduced to matching  covered graphs.  A subgraph  $H$ of
$G$ is \emph{conformal} if~$H$ and~$G - V(H)$ are matchable.

Let $D$ and $D'$ be orientations of  the same graph $G$.  Denote by $D
\symdiff D'$ the set of edges  of $G$ that have different orientations
in  $D$  and  in~$D'$.   The  next  result  follows  easily  from  the
definition of sign of a perfect matching.
\begin{proposition}\label{prop:sign-inversion}
	Let $G$  be a graph  and let $D$  and $D'$ be  orientations of
	$G$.  Let $M$ be a perfect matching of $G$.  Then,
		\[ \sign_D(M) \cdot \sign_{D'}(M) = (-1)^{|M \cap (D \symdiff D')|}. \tag*{\qed} \]
\end{proposition}

Let $T$  be a trail  of $G$.   An arc of  $D$ is a  \emph{forward arc}
of~$T$ if it is traversed from tail  to head.  The set of forward arcs
of $T$ is denoted by $\fw_D(T)$.  Let~$Q$ be a conformal cycle of $G$.
The \emph{sign}  of $Q$  in $D$,  denoted by  $\sign_D(Q)$, is  $1$ if
$|\fw_D(Q)|$ is odd  and $-1$ otherwise.  As $Q$ is  conformal then it
is even, hence the sign of $Q$ in $D$ does not depend on the direction
of its traversal.   The following basic result appears in  the book by
Lovász and Plummer~\cite[Chapter 8]{lopl86}.
\begin{proposition}\label{prop:sign-product}
	Let $G$ be a graph and let  $D$ be an orientation of $G$.  Let
	$M$ and $N$ be perfect matchings of $G$.  Then,
\[ \sign_D(M) \cdot \sign_D(N) = \prod_{Q \in \mathcal{C}} \sign_D(Q), \]
	where  $\mathcal{C}$ is  the set  of all  $(M, N)$-alternating
	cycles.  \qed
\end{proposition}

\paragraph*{Similarity of Orientations}

Let $X$ be a set of vertices of $G$.  The \emph{cut} with \emph{shore}
$X$ of  $G$, denoted by $\partial(X)$,  is the set of  edges that have
one   endpoint  in~$X$   and  the   other  in   $\overline{X}$,  where
$\overline{X}$ is the set $V(G) - X$.

Let $D$ be an orientation of $G$ and  let $F$ be a set of arcs of $D$.
We denote by $D \rev F$ the orientation obtained from $D$ by reversing
the orientation of all the arcs of $F$.  Two orientations $D$ and $D'$
are \emph{similar} if $D' = D \rev C$, where $C$ is a cut of $G$.  Two
$k$-orientations $\mathbf{D}$ and  $\mathbf{D}'$ are \emph{similar} if
$D_i$ is similar to $D'_i$ for every $i$, $1 \leq i \leq k$.  The next
result follows from the fact that the parity of the number of edges of
any perfect matching in a cut  $\partial(X)$ is equal to the parity of
$|X|$.

\begin{proposition}
	Let $D$ be an orientation of a graph $G$, let $M$ be a perfect
	matching of  $G$ and let $C  := \partial(X)$ be a  cut of $G$.
	Let $D' = D \rev C$.  Then,
		\[ \sign_D(M) \cdot \sign_{D'}(M) = (-1)^{|X|}. \tag*{\qed} \]
\end{proposition}

Let $\mathbf{D}$  and $\mathbf{D'}$  be two  similar $k$-orientations.
If  $\mathbf{D}$   is  pfaffian   with  solution   $\boldalpha$,  then
$\mathbf{D'}$  is  pfaffian.   Indeed,  we  can  modify  $\boldalpha$,
possibly  changing the  sign  of  some of  its  entries,  to obtain  a
solution of $\mathbf{D}'$.

\begin{corollary}\label{cor:similar-k-orientations}
	Let   $\mathbf{D}$   and    $\mathbf{D'}$   be   two   similar
	$k$-orientations.    The   $k$-orientation   $\mathbf{D}$   is
	pfaffian if and only if $\mathbf{D}'$ is pfaffian.  \qed
\end{corollary}

\paragraph*{Cuts and Contractions}

Let  $X$ be  a set  of vertices  of $G$.   Denote by  $G/X$ the  graph
obtained from $G$ by the contraction of the vertices in the set $X$ to
a  single  vertex  and  removing  the  resulting  loops.   Let  $C  :=
\partial(X)$.  The \emph{$C$-contractions} of $G$ are the graphs $G/X$
and $G/\overline{X}$.

Let $C$  be a  cut of a  matching covered graph  $G$.  The  cut~$C$ is
\emph{separating} if both \mbox{$C$-contractions}  of $G$ are matching
covered.  The cut  $C$ is \emph{tight} if  $|M \cap C| =  1$ for every
perfect  matching  $M$  of  $G$.  The  following  characterization  of
separating   cuts    appears   in    the   book   by    Lucchesi   and
Murty~\cite[Chapter~4]{lm2024}.
\begin{theorem}\label{thm:separating-cut-characterization}
	A cut $C$ of a matching covered graph $G$ is separating if and
	only  if each  edge of  $G$ lies  in a  perfect matching  that
	contains precisely one edge in $C$.  \qed
\end{theorem}

It is then easy  to see that every tight cut is  separating.  A cut is
\emph{trivial}  if there  is a  single vertex  in one  of its  shores.
Notice that every trivial cut is also tight.  A matching covered graph
free  of   non-trivial  tight  cuts   is  a  \emph{brick}  if   it  is
non-bipartite and is a \emph{brace} if it is bipartite.  A \emph{tight
  cut decomposition} of  a graph is a collection of  bricks and braces
obtained by  recursively applying  a $C$-contraction of  a non-trivial
tight cut $C$ whenever possible.  In 1987, Lovász~\cite{lova87} proved
that every  tight cut decomposition of  a graph is, in  a certain way,
unique.

\begin{theorem}[Lovász,~\cite{lova87}]
	Every application of the tight cut decomposition to a matching
	covered graph $G$ produces the  same set of bricks and braces,
	up to multiple edges.  \qed
\end{theorem}

\paragraph*{Linear Algebra}

The following basic result, which is a consequence of the Rank-Nullity
Theorem, will be useful in Section~\ref{sec:conformal-subgraphs}.
\begin{lemma}\label{lem:rank-upperbound}
	Let $\mathbf{A}$ be a matrix with $n$ columns.  Let $\mathbf{A
	  x} = \mathbf{b}$ be a linear system.  If this system has $t$
	linearly independent solutions, then $\rank(\mathbf{A}) \leq n
	- t + 1$.  \qed
\end{lemma}

The \emph{Hadamard product} of two vectors $\mathbf{u}, \mathbf{v} \in
\mathbb{R}^n$, denoted by $\mathbf{u} \odot \mathbf{v}$, is the vector
obtained   by  position-wise   product.   Namely,   $\mathbf{u}  \odot
\mathbf{v} = (u_i \cdot v_i)_{1 \leq i \leq n}$.  Let $\mathcal{S}$ be
a  set  of vectors in $\mathbb{R}^n$.   The  Hadamard  product  of  all the  vectors  in
$\mathcal{S}$  is  denoted   by  $\prod_{\mathbf{v}  \in  \mathcal{S}}
\mathbf{v}$.

We denote the  $i$-th row of a matrix  $\mathbf{A}$ by $\mathbf{A}_i$.
The  \emph{Khatri-Rao   product}  of  two  matrices   $\mathbf{A}  \in
\mathbb{R}^{m_1 \times n}$ and  $\mathbf{B} \in \mathbb{R}^{m_2 \times
  n}$, denoted by  $\mathbf{A} \ast \mathbf{B}$, is  a matrix obtained
by  the  Hadamard product  of  pairs  of  rows from  $\mathbf{A}$  and
$\mathbf{B}$.  Formally,  the product $\mathbf{A} \ast  \mathbf{B}$ is
the matrix  $\mathbf{C} \in  \mathbb{R}^{(m_1m_2)\times n}$  such that
$\mathbf{C}_{(i - 1)\cdot m_2 +  j} = \mathbf{A}_i \odot \mathbf{B}_j$
for  every $i$, $1 \leq i \leq m_1$,  and every $j$,  $1 \leq j \leq m_2$.

It turns out that the products defined above are useful for describing
properties of signs of matchings  and cycles in $k$-orientations.
For example, we  can easily generalize  Proposition~\ref{prop:sign-product} to
$k$-orientations.
Let
$\mathbf{D}$  be  a $k$-orientation  of  $G$.   The \emph{sign}  of  a
conformal cycle $Q$ in $\mathbf{D}$ is
	\[\boldsign_\mathbf{D}(Q) := (\sign_{D_1}(Q), \sign_{D_2}(Q), \dots, \sign_{D_k}(Q)).\]
\begin{proposition}\label{prop:boldsign-product}
	Let $G$ be  a graph and let $\mathbf{D}$  be a $k$-orientation
	of $G$.  Let $M$ and $N$ be perfect matchings of $G$.  Then,
		\[ \boldsign_\mathbf{D}(M) \odot \boldsign_\mathbf{D}(N) = \prod_{Q \in \mathcal{C}} \boldsign_\mathbf{D}(Q),\]
	where  $\mathcal{C}$ is  the set  of all  $(M, N)$-alternating
	cycles.  \qed
\end{proposition}

\subsection{Outline of the Paper}

In  Section~\ref{sec:related-work},  we  review a  recently  developed
method for  counting perfect matchings: the  symbolic pfaffian method.
The known results of lower bounds for this method are similar to those
obtained  in this  paper  for  the pfaffian  number.   We discuss  the
similarities  of the  symbolic  pfaffian method  and the  $k$-pfaffian
method.  Also, we  show some implications of the lower  bounds for the
symbolic   pfaffian    method   to    the   pfaffian    numbers.    In
Section~\ref{sec:conformal-subgraphs},  we present  a lower  bound for
the pfaffian number of a graph  related to the pfaffian numbers of its
conformal subgraphs.  As an intermediate step, we prove an upper bound
for the  rank of two matrices  related to their Khatri-Rao  product, a
result of independent relevance.  Then,  we conclude that the pfaffian
number  is unbounded  by showing  an  infinite family  of graphs  with
increasing pfaffian number.   In Section~\ref{sec:separating-cuts}, we
show that  the pfaffian  number of  a graph is  at least  the pfaffian
number of its separating cut contractions.  Also,
we  discuss the  relation  of  this result  with  previous results  on
pfaffian   graphs.   Finally,   in  Section~\ref{sec:conclusion},   we
conclude with some final remarks about these new results.

\section{Related Work: The Symbolic Pfaffian Method}
\label{sec:related-work}

In 2000, Tesler~\cite{tesl00} proved the assertion of Kasteleyn that implies that graphs embeddable in orientable sufaces of genus $g$ are $4^g$-pfaffian.
The prove by Tesler, as he noticed, also implied that the number of perfect matchings of graphs embeddable in surfaces of genus~$g$ can be computed via evaluation of the pfaffian of a matrix with symbolic entries.
The   symbolic  pfaffian   method for counting perfect matchings  was then formally described in 2017,  by   Babenko  and
Vyalyi~\cite{bavy17}, for bipartite  graphs.  Then,  in 2021,
Vyalyi~\cite{vyalyi21} generalized  this method for  arbitrary graphs.
Some of the  notation below differs from the notation of Babenko and Vyalyi for the sake of
comparison with the $k$-pfaffian method.

Let $G$  be a graph and  let $d$ be an  integer.
Consider the vector space $\gf(2)^d$ formed by the binary vectors of dimension $d$.
Denote the sum in this vector space by $\oplus$.
A \emph{$d$-symbolic function}  $\tau \colon E(G) \rightarrow \gf(2)^d$ of $G$ is a function defined in the edges of $G$ and that extends to perfect matchings by $\tau(M) =  \bigoplus_{e \in  M} \tau(e)$ for every $M \in \mathcal{M}(G)$.
A \emph{$d$-symbolic orientation}  of~$G$ is a pair
$(D,  \tau)$, where  $D$  is an  orientation of~$G$  and  $\tau$ is  a
$d$-symbolic  function   of  $G$.    A  $d$-symbolic   orientation  is
\emph{pfaffian} if, for every pair  of perfect matchings $M$ and $N$ of~$G$,
	\[\sign_D(M) \neq \sign_D(N) \implies  \tau(M)  \neq \tau(N).\]
Every graph~$G$ admits a pfaffian $|E(G)|$-symbolic orientation.
Indeed, a \mbox{$|E(G)|$-symbolic} function which maps every edge to a different vector in the standard basis implies that $\tau(M) \neq \tau(N)$ for every two different perfect matchings $M$ and $N$ of $G$.
The
\emph{symbolic pfaffian number}  of~$G$, denoted by  $\spf(G)$, is the
smallest  integer $d$  such that  $G$ admits  a pfaffian  $d$-symbolic
orientation.   In  particular,  the  symbolic  pfaffian  number  of  a
pfaffian graph is 0.

Given a pfaffian  $d$-symbolic orientation $(D, \tau)$ of  $G$, we can
count  the number  of perfect  matchings  of $G$  with a  parameterized
algorithm.    Let  $\mathcal{R}_d   =   \mathbb{Z}[t_1,  t_2,   \dots,
  t_d]/(t_i^2  - 1)$  be a polynomial  ring in  $d$ variables  with
exponents modulo  2.  Associate  to every  edge of  $G$ a  monomial in
$\mathcal{R}_d$ with coefficient 1  and exponents corresponding to the
entries  of $\tau(e)$.   Let~$\boldtau{(D)}$  be the  skew-symmetric
matrix obtained by multiplying  the nonzero entries of~$\mathbf{A}_\mathbf{D}$
by the monomials associated with the corresponding edges.  The number of
perfect matchings of $G$ is equal to the sum of the absolute values of
the coefficients of $\Pf(\boldtau{(D)})$.   
Notice, however, that determinant-based algorithms for computing the pfaffian are not suitable for matrices over $R_d$.
Nevertheless, there  are efficient combinatorial algorithms for
computing    the   pfaffian  directly and   without using divisions~\cite{msv99}.
As the intermediate operations of such algorithms
are executed over polynomials with at most~$2^d$ terms, the complexity
of this method is $\mathcal{O}(2^d \mathrm{poly}(n))$.

Babenko  and  Vyalyi~\cite{bavy17} also  extended  this  method for  a
polynomial  ring  without  restrictions  in  the  exponents  of  their
variables.  However,  the most  interesting results  are for  the case
described here, with exponents modulo 2.  Also, Vyalyi~\cite{vyalyi21}
considered a  generalization of  this method that  uses more  than one
symbolic orientation.   The results  obtained for  that generalization
hold, in particular, for the version described here.

In 2021, Vyalyi~\cite{vyalyi21} obtained lower bounds for the symbolic
pfaffian number of  arbitrary graphs.  Let $\Pf^*(G)$  be the greatest
absolute  value  of  $\Pf(\mathbf{A_D})$  among  all  orientations~$D$
of~$G$.
\begin{theorem}[Vyalyi,~\cite{vyalyi21}]\label{thm:vyalyi-lower-bound}
	Let $G$ be a graph.  Then,
		\[ \spf(G) \geq \dfrac{1}{2}\log_2\dfrac{|\mathcal{M}(G)|}{\Pf^*(G)}. \tag*{\qed} \]
\end{theorem}
\begin{corollary}[Vyalyi,~\cite{vyalyi21}]\label{cor:vyalyi-complete-graphs}
	$\spf(K_{2n}) = \Omega(n \lg n)$ \qed
\end{corollary}

We noticed that we can build  a pfaffian $d$-symbolic orientation of a
graph given a  pfaffian $(d + 1)$-orientation.  This  implies that the
pfaffian number of  a graph is lower bounded by  its symbolic pfaffian
number as we prove below.

\begin{theorem}\label{thm:pf-and-spf}
	Let $G$ be a graph.  Then, $\pf(G) \geq \spf(G) + 1$.
\end{theorem}
\begin{proof}
	Let  $k  :=  \pf(G)$.   
	If $k = 1$, then $G$ is pfaffian, $\spf(G) = 0$ and the theorem holds trivially.
	Suppose then that $G$ is non-pfaffian and hence $k > 1$.
	Let  $\mathbf{D}$  be  a  $k$-pfaffian
	orientation  of~$G$.  Let  $\tau$  be  a $(k  -  1)$-symbolic
	function of $G$ defined by
		\[ 
		\tau(e)_i = \left\{
			\begin{array}{rl}
				1, & \text{if } e \in D_i \symdiff D_k \\  
				0, & \text{otherwise} \\
			\end{array}
		\right.
		\]
	for every edge $e \in E(G)$ and $i$, $1 \leq i \leq k - 1$.  Then,
	by definition, for every perfect  matching~$M$ of $G$, we have
	that $\tau(M)_i \equiv  |M \cap (D_i \symdiff  D_k)| \pmod 2$.
	We will  show that the  $(k - 1)$-symbolic  orientation $(D_k,
	\tau)$ is pfaffian.

	Since $G$ is not pfaffian, there are perfect matchings of $G$, say $M$ and $N$, with different signs in~$D_k$.
	Assume, by  contradiction, that  $\tau(M) =
	\tau(N)$.  Then, by the construction of $\tau$, we have that
		\[ |M \cap (D_i \symdiff D_k)| \equiv |N \cap (D_i \symdiff D_k)| \pmod 2, \]
	for    every   $i$, $1 \leq i \leq k - 1$.    Therefore,    by
	Proposition~\ref{prop:sign-inversion}, we obtain that
		\[ \sign_{D_i}(M) \cdot \sign_{D_k}(M) = \sign_{D_i}(N) \cdot \sign_{D_k}(N). \]
	Given that the signs of $M$ and $N$ in $D_k$ are different, we
	know that  $\sign_{D_i}(M) \neq \sign_{D_i}(N)$.   We conclude
	that  the signs  of  $M$  and $N$  are  different  in all  the
	orientations of $\mathbf{D}$.  However,  this implies that the
	equation $\boldsign_\mathbf{D} \mathbf{x} = \mathbf{1}$ has no
	solution       because        $\boldsign_\mathbf{D}(M)       +
	\boldsign_\mathbf{D}(N) = 0$.  
	This contradicts that $\mathbf{D}$ is pfaffian.
\end{proof}

Applying  Corollary~\ref{cor:vyalyi-complete-graphs},  we  obtain  the
following lower bound for the pfaffian number.
\begin{corollary}
	$\pf(K_{2n}) = \Omega(n \lg n)$ \qed
\end{corollary}

Vyalyi~\cite{vyalyi21} also proved lower bounds for a family of sparse
graphs.     From    the    results    of    Vyalyi    together    with
Theorem~\ref{thm:pf-and-spf}, we conclude that  the pfaffian number is
unbounded.  In  the next section,  we show  a different proof  of that
fact  that does  not  depend  on the  lower  bounds  for the  symbolic
pfaffian method.

\section{The Pfaffian Number and Conformal Subgraphs}
\label{sec:conformal-subgraphs}

In this section, we  prove a lower bound for the  pfaffian number of a
graph that depends on the  pfaffian number of its conformal subgraphs.
This  lower bound  is a  direct  consequence of  a new  result on  the
Khatri-Rao product.  In  that sense, we start  this section describing
the relation between  the Khatri-Rao product and the  restriction of a
$k$-orientation to conformal subgraphs.

\begin{proposition}\label{prop:khatri-rao-and-signature-matrix}
	Let  $\mathbf{D}$ be  a $k$-orientation  of a  matching covered graph
	$G$.  Let $H'$ be a non-spanning conformal subgraph of $G$ and
	let $H'' := G -  V(H')$.  Let $\mathbf{D}'$ and $\mathbf{D}''$
	be  the  restrictions  of  $\mathbf{D}$  to  $H'$  and  $H''$,
	respectively.  Then, the  matrix $\boldsign_{\mathbf{D}'} \ast
	\boldsign_{\mathbf{D}''}$ is formed by a subset of the rows of
	$\boldsign_\mathbf{D}$.
\end{proposition}
\begin{proof}
	Let  $M'$ and  $M''$  be  perfect matchings  of  $H'$ and  $H''$,
	respectively.  
	Adjust the  enumeration of the  vertices of $G$ such  that the
	first~$2|M'|$ vertices are the vertices in $H'$.
	For every $i$, $1 \leq i \leq k$, adjust the order of the edges listed in
	the  permutation~$\pi_{D_i}(M' \cup  M'')$ such  that all  the
	edges in $M'$  are listed first.  Hence, it is  clear that the
	number of inversions  of $\pi_{D_i}(M' \cup M'')$  is equal to
	the  sum  of  the  number  of  inversions  of~$\pi_{D'_i}(M')$
	and of $\pi_{D''_i}(M'')$.      Thus,      we     obtain     that
	 \[\sign_{D'_i}(M')  \cdot  \sign_{D''_i}(M'') =  \sign_{D_i}(M'\cup  M'').\]
	Therefore,  $\boldsign_{\mathbf{D}'}(M')  \odot
	\boldsign_{\mathbf{D}''}(M'') = \boldsign_{\mathbf{D}}(M' \cup
	M'')$ holds by the definition of the Hadamard product.
	Given that $M'$ and $M''$ are arbitrary perfect matchings
	and by the  definition of the Khatri-Rao  product, we conclude
	that  every row  of the  matrix $\boldsign_{\mathbf{D'}}  \ast
	\boldsign_{\mathbf{D''}}$       is        a       row       of
	$\boldsign_{\mathbf{D}}$.
\end{proof}

Now we  will prove  an upper  bound for the  sum of  the ranks  of two
particular matrices  using the  Khatri-Rao product.   We will  use the
following straightforward result.
\begin{lemma}\label{lem:linear-independence}
	Let  $\mathcal{B}   =  \{\mathbf{v_1},   \mathbf{v_2},  \dots,
	\mathbf{v_n}\}$ be  a set  of linearly independent  vectors in
	$\mathbb{R}^k$.     Let~$\mathbf{u}$    be   a    vector    in
	$\mathbb{R}^k$ with no zero  entries.  The set $\mathcal{B}' =
	\{\mathbf{u}  \odot  \mathbf{v_i}\}_{1  \leq  i  \leq  n}$  is
	linearly independent.  \qed
\end{lemma}

\begin{theorem}\label{thm:khatri-rao-upperbound}
	Let $\mathbf{A} \in \mathbb{R}^{m_1 \times n}$ and $\mathbf{B}
	\in  \mathbb{R}^{m_2  \times n}$  be  matrices  with the  same
	number of columns.  If there  is a vector $\boldalpha$ without
	zero  entries  that  satisfies $(\mathbf{A}  \ast  \mathbf{B})
	\boldalpha = \mathbf{1}$, then
		\[ \rank(\mathbf{A}) + \rank(\mathbf{B}) - 1 \leq n. \]
\end{theorem}
\begin{proof}
	Suppose that such a vector  $\boldalpha$ exists.  Then, by the
	definition   of  the   Khatri-Rao   product,   we  have   that
	$(\mathbf{A}_i \odot \mathbf{B}_j) \boldalpha = 1$, for
	every $i$, $1 \leq i \leq m_1$, and every $j$, $1 \leq j \leq m_2$.   
	Notice that, after transposing $\mathbf{B}_j$ to match dimensions, we
	can  apply  the  associativity between  the  Hadamard
	product  and  the  matrix product, yielding $\mathbf{A}_i
	(\mathbf{B}^T_j  \odot  \boldalpha)   =  1$.   Therefore,
	the vector set $\{\mathbf{B}^T_j \odot  \boldalpha\}_{1 \leq j \leq  m_2}$ is a
	set of  solutions to  the system~$\mathbf{A x}  = \mathbf{1}$.
	Since the vector $\boldalpha$ does  not contain zeroes and due
	to   Lemma~\ref{lem:linear-independence},   we   obtain   that
	$\mathbf{A x}  = \mathbf{1}$ has  at least~$\rank(\mathbf{B})$
	linearly          independent          solutions.           By
	Lemma~\ref{lem:rank-upperbound},     we      conclude     that
	$\rank(\mathbf{A}) \leq n - \rank(\mathbf{B}) + 1$.
\end{proof}

The theorem above  has direct consequences for the  pfaffian number of
graphs. 

\begin{corollary}
	Let  $G$ be  a matching covered graph.  Let  $H'$ be  a non-spanning
	conformal subgraph  of $G$ and  let $H''  := G -  V(H')$.  Let
	$\mathbf{D}$ be  a pfaffian $\pf(G)$-orientation of  $G$.  Let
	$\mathbf{D}'$  and  $\mathbf{D}''$   be  the  restrictions  of
	$\mathbf{D}$ to $H'$ and $H''$, respectively.  Then,
		\[ \rank(\boldsign_{\mathbf{D}'}) + \rank(\boldsign_\mathbf{D''}) - 1 \leq \pf(G). \]
\end{corollary}
\begin{proof}
	Let   $\boldalpha$  be   a  solution   of  $\mathbf{D}$.    By
	Proposition~\ref{prop:khatri-rao-and-signature-matrix},     we
	know    that   $\boldalpha$    is    a    solution   to    the
	system~$(\boldsign_\mathbf{D'}   \ast  \boldsign_\mathbf{D''})
	\mathbf{x}  = \mathbf{1}$.   Notice that  $\boldalpha$ has  no
	zeroes;  otherwise,  we  could  discard  the  orientations  of
	$\mathbf{D}$ corresponding to the zero entries of $\boldalpha$
	and  obtain  a  pfaffian  $k$-orientation of  $G$  with  $k  <
	\pf(G)$, a  contradiction.  Thus, this corollary  follows from
	Theorem~\ref{thm:khatri-rao-upperbound}.
\end{proof}

The obtained lower  bound for the pfaffian number can  also be written
using the pfaffian numbers of  the subgraphs by noticing that the
rank of a signature matrix of a pfaffian $k$-orientation is lower bounded by  the pfaffian
number of the underlying graph.

\begin{proposition}
	Let $\mathbf{D}$ be a pfaffian $k$-orientation of a graph $G$.
	Then,
		\[ \pf(G) \leq \rank(\boldsign_\mathbf{D}). \]
\end{proposition}
\begin{proof}
	Let $\mathbf{A} :=  \boldsign_\mathbf{D}$.  Since $\mathbf{D}$
	is pfaffian,  the vector $\mathbf{1}$ is  a linear combination
	of  the  columns  of  $\mathbf{A}$.  Let  $\mathbf{D}'$  be  a
	$\rank(\mathbf{A})$-orientation formed  by $\rank(\mathbf{A})$
	orientations   of  $\mathbf{D}$   corresponding  to   linearly
	independent columns  of $\mathbf{A}$.   Clearly, $\mathbf{D}'$
	is  pfaffian.  Therefore,  $\rank(\mathbf{A})$  cannot be  less
	than $\pf(G)$.
\end{proof}
\begin{corollary}\label{cor:pf-conformal-lower-bound}
	Let  $G$ be  a matching covered graph.  Let  $H'$ be  a non-spanning
	conformal subgraph of $G$ and let $H'' := G - V(H')$.  Then,
		\[ \pf(H') + \pf(H'') - 1 \leq \pf(G). \tag*{\qed} \]
\end{corollary}

To apply this lower bound,  we need a non-spanning subgraph.  However,
the restriction of a pfaffian $k$-orientation of a graph $G$ to a spanning
subgraph of $G$ clearly remains pfaffian.  
This observation allows us to extend the lower bound to arbitrary conformal subgraphs, yielding the following corollary.
\begin{corollary}
	Let $G$ be a matching covered graph.   If $H$ is a conformal subgraph
	of $G$, then the inequality $\pf(G) \geq \pf(H)$ holds.  \qed
\end{corollary}

The  lower  bound in  Corollary~\ref{cor:pf-conformal-lower-bound}  is
sharp.  Indeed, consider the graph  $K_{4,4}$.  Let $H$ be a conformal
subgraph of $K_{4, 4}$ isomorphic  to $K_{3,3}$.  The graph $K_{4,4} -
V(H)$ is formed by a single edge and, hence, is pfaffian.  Then, given
that         $\pf(K_{3,3})         =         4$         and         by
Corollary~\ref{cor:pf-conformal-lower-bound},    we     obtain    that
$\pf(K_{4,4}) \geq 4$.  Moreover, since $K_{4,4}$ is embeddable in the
torus   and   by    Theorem~\ref{thm:pf-genus},   we   conclude   that
$\pf(K_{4,4}) = 4$ This example is  a particular case of the following
corollary.
\begin{corollary}
	Let  $G$ be  a matching covered graph.   Let $H$  be a  non-spanning
	conformal subgraph  of $G$.   If $\pf(G)  = \pf(H)$,  then the
	graph $G - V(H)$ is pfaffian.  \qed
\end{corollary}

We      do not      know      if      the      lower      bound      in
Corollary~\ref{cor:pf-conformal-lower-bound} remains sharp  if we only
consider non-pfaffian conformal subgraphs.

\subsection{Unboundedness of the Pfaffian Number}

The  lower   bound  from  Corollary~\ref{cor:pf-conformal-lower-bound}
implies that the  pfaffian number of graphs is  unbounded.  Indeed, we
only  need  to  consider  a  graph  with  many  disjoint  non-pfaffian
conformal  subgraphs.  For  example,  we can  consider  the family  of
bipartite  complete  graphs $K_{3n,  3n}$.   A  graph in  this  family
contains   $n$  disjoint   copies  of   $K_{3,3}$.   Hence,   applying
Corollary~\ref{cor:pf-conformal-lower-bound}  recursively,  we  obtain
that $\pf(K_{3n, 3n})  \geq n \cdot \pf(K_{3,3}) - (n  - 1) = 3n  + 1$.  This
graph  has $9n^2$  edges and,  therefore, the  pfaffian number  of the
bipartite complete graphs grows at least as fast as the square root of
their number of edges.

We also  constructed an  infinite family  of sparse  matchings covered
graphs in which  the pfaffian number grows at least  linearly with the
number of  edges.  Then,  we noticed that  Vyalyi~\cite{vyalyi21} also
constructed a  similar family to  prove lower bounds for  the symbolic
pfaffian method in sparse graphs.  We will focus on the Vyalyi family
of graphs because these graphs are  also cubic.  We will show that
the pfaffian numbers of the graphs  in this family grow linearly using
Corollary~\ref{cor:pf-conformal-lower-bound}.

\paragraph*{The Vyalyi Family of Graphs}

An \emph{ear} in a graph is a path of odd length all of whose internal
vertices have degree two.  A \emph{bisubdivision}  of an edge $e = uv$
of a graph $G$ is a graph obtained by replacing $e$ by an ear from $u$
to $v$.  A \emph{bisubdivision} of a  graph $G$ is a graph obtained by
the bisubdivision  of some of its  edges.  Let $H$ be  a bisubdivision
of~$K_{3,3}$ in which  exactly two non-adjacent edges  are replaced by
ears of length three.  The graphs in the Vyalyi family contain several
copies  of $H$.   Denote by  $\mathbb{V}_n$ the  member of  the Vyalyi
family that contains exactly $n$ copies  of $H$, denoted by $H_1, H_2,
\dots, H_n$.  Denote  by~$u_i$ and~$v_i$ the interior  vertices of one
of the ears  of length three in $H_i$; similarly,  denote by $p_i$ and
$q_i$ the interior vertices of the other ear of length three in~$H_i$.
The graph $\mathbb{V}_n$  is formed by the graphs $H_i$  and the edges
$u_ip_{i + 1}$ and  $v_iq_{i + 1}$, for every~$i$, $1  \leq i \leq n - 1$,
and the  edges $u_np_1$ and $v_nq_1$.   See Figure~\ref{fig:vyalyi-2}.
Notice that  the graphs  in this family  are bipartite,  cubic and
matching covered.
\begin{figure}[tb]
	\centering \includegraphics[width=5cm]{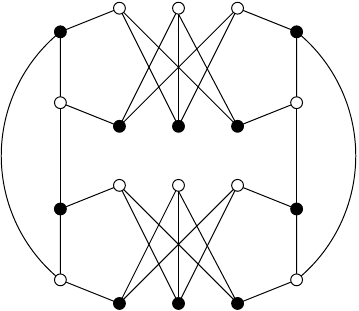}
	\caption{ The graph $\mathbb{V}_2$.
		\label{fig:vyalyi-2}
	}
\end{figure}

To  apply Corollary~\ref{cor:pf-conformal-lower-bound}  to the  Vyalyi
family,  we will  use the  following fact:  the pfaffian  number of  a
bisubdivision of  a graph $G$ equals  the pfaffian number of  $G$ (see
Proposition~\ref{prop:pf-retract} in the next section).  Then, the pfaffian
number of  each conformal subgraph  $H_i$ of $\mathbb{V}_n$  is four.
By applying  Corollary~\ref{cor:pf-conformal-lower-bound} recursively,
we   conclude  that   $\pf(\mathbb{V}_n)   \geq  3n   +  1$.    Using
Theorem~\ref{thm:vyalyi-lower-bound},   Vyalyi~\cite{vyalyi21}  proved
that $\spf(\mathbb{V}_n) \geq n \log_2 \sqrt{3/2}$.  With that result
and Theorem~\ref{thm:pf-and-spf},  we obtain  that $\pf(\mathbb{V}_n)
\geq n  \log_2 \sqrt{3/2}  + 1$.   The lower  bound proved  here using
Corollary~\ref{cor:pf-conformal-lower-bound} is slightly better.

\section{The Pfaffian Number and Separating Cuts}
\label{sec:separating-cuts}

In  1991,  Little  and  Rendl~\cite{lire91}  reduced  the  problem  of
deciding if a graph is pfaffian to deciding if the graphs in its tight
cut  decomposition  are  pfaffian.   Their result  can  be  stated  as
follows.

\begin{theorem}[Little and Rendl,~\cite{lire91}]\label{thm:pfaffian-tight-cut}
	Let $G$ be a matching covered graph and let $C$ be a tight cut
	of  $G$.  The  graph  $G$  is pfaffian  if  and  only if  both
	$C$-contractions of $G$ are pfaffian.  \qed
\end{theorem}

The sufficient condition of the theorem above does not hold if the cut is separating but not tight.
As an example, consider a cut $C$ of the Petersen graph whose shores are the $5$-cycles.
The cut $C$ is separating and both $C$-contractions are pfaffian because they are planar.
However, the Petersen graph is not pfaffian.
In contrast, the necessary condition of Theorem~\ref{thm:pfaffian-tight-cut} still holds for non-tight separating cuts.

\begin{theorem}[Little and Rendl,~\cite{lire91}]
	Let  $G$  be  a  matching  covered graph  and  let  $C$  be  a
	separating cut of $G$.  If  one of the $C$-contractions of $G$
	is not pfaffian, then $G$ is not pfaffian.  \qed
\end{theorem}

In 2012, Carvalho, Lucchesi and Murty~\cite{clm12} proved a stronger version of the theorem above.
\begin{theorem}[Carvalho, Lucchesi and Murty,~\cite{clm12}]
	Let $C$ be a separating cut of a matching covered graph $G$ and let $D_0$ be a pfaffian orientation of $G$.
	Then $G$ has an orientation $D$ similar to $D_0$ such that both $C$-contractions of $D$ are pfaffian.
	\qed
\end{theorem}

We prove a generalization of the previous theorem to \mbox{$k$-orientations}.

\begin{theorem}\label{thm:clm-generalization}
	Let $C$ be a separating cut of a matching covered graph $G$ and let $\mathbf{D^{(0)}}$ be a pfaffian $\pf(G)$-orientation of $G$.
	Then $G$ has an $\pf(G)$-orientation $\mathbf{D}$ similar to $\mathbf{D^{(0)}}$ such that both $C$-contractions of $\mathbf{D}$ are pfaffian.
\end{theorem}
\begin{proof}
	We will build $\mathbf{D}$ and then show that the $C$-contractions of $\mathbf{D}$ remain pfaffian.
	For this purpose, we will describe a method to obtain vertex sets $S_i$ for every $i$, $1 \leq i \leq \pf(G)$, and show that the $\pf(G)$-orientation $\mathbf{D}$ defined by $D_i := D^{(0)}_i \rev \partial(S_i)$ is as desired.
	Let $G' := G / \overline{X}$ and let $G'' :=  G/X$.
	Let~$e  := uv$ be an    edge   of    $C$    with   $u    \in    X$.
	Given that $C$ is a separating cut, there is a perfect matching, say $M$, such that $M \cap C = \{e\}$.
	Let $M'$ and $M''$ be the restrictions of $M$ to $G'$ and to $G''$, respectively.

	Let $V'$ be the set of vertex of $X$ that are ends of the edges in $C$.
	Likewise, let $V''$ be the set of vertex of~$\overline{X}$ that are ends of the edges in $C$. 
	For any $w'' \in V''$, we define $P(w'')$ as an even length $M''$-alternating path from $v$ to $w''$ as follows.
	If $w'' = v$, then $P(w'') := (v)$. Otherwise, there is an edge $f \neq e$ in $C$ incident to $w''$.
	As $C$ is separating, there is a perfect matching~$M_f$ of~$G$ such that $M_f \cap C = \{f\}$.
	Therefore, there is an $(M, M_f)$-alternating cycle $Q$ of $G$ such that $E(Q) \cap C = \{e, f\}$.
	Define $P(w'')$ as a $(M'', M_f)$-alternating path from $v$ to $w''$ in $\overline{X}$.
	See Figure~\ref{fig:separating-cut-proof}.
	Indeed, the path $P(w'')$ has even length because the first edge of the path is in $M_f$ and the last edge is in $M''$.
	Likewise, for every $w' \in V'$, define $P(w')$ as an even length $M'$-alternating path from $u$ to $w'$.
	For every $i$, $1 \leq i \leq \pf(G)$, define
		\[S_i := \{w \in V' \cup V'' \colon |\fw_{D^{(0)}_i}(P(w))| \equiv 1 \pmod 2 \}. \]
	Since the paths $P(w)$ have even length, the sets $S_i$ are well-defined.

	\begin{figure}[tb]
		\centering \includegraphics{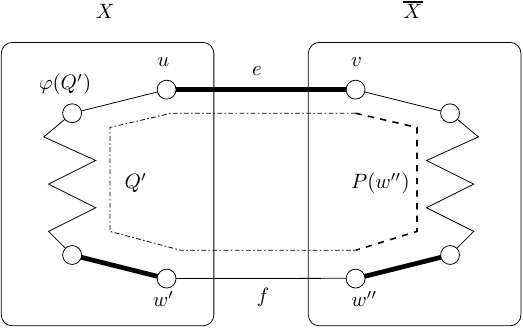}
		\caption{                   Proof                   of
		  Theorem~\ref{thm:clm-generalization}.
			\label{fig:separating-cut-proof}
		}
	\end{figure}

	Recall that \mbox{$D_i = D^{(0)}_i \rev \partial(S_i)$}.
	The $\pf(G)$-orientations $\mathbf{D}$ and $\mathbf{D^{(0)}}$ are similar; hence, by Corollary~\ref{cor:similar-k-orientations}, $\mathbf{D}$ is pfaffian.
	Also, by the definition of the sets $S_i$, we have that 
	\begin{equation}\label{eq:evenly-oriented-paths}
		|\fw_{D_i}(P(w))| \equiv 0 \pmod 2, \text{ for every } w \in V' \cup V''.
	\end{equation}
	We will show that the $C$-contractions of $\mathbf{D}$ are pfaffian.

	Let $\mathbf{D'}$ be the contraction of $\mathbf{D}$ to $G'$.
	For every $M'$-alternating cycle $Q'$ in $G'$, let $\varphi(Q')$ be the $M$-alternating cycle of $G$ defined as follows.
	If $E(Q') \cap C$ is empty, then $\varphi(Q') := Q'$.
	Otherwise, there is a unique edge $f \neq e$ in $C$ that is contained in $E(Q')$.
	Let $w''$ be the end of $f$ that is contained in $V''$.
	Define $\varphi(Q')$ as the $M$-alternating cycle formed by the edges in $E(Q') \cup E(P(w''))$.
	See Figure~\ref{fig:separating-cut-proof}.

	\begin{claim}\label{claim:equal-signs}
		For every $M'$-alternating cycle $Q'$, $\boldsign_\mathbf{D'}(Q') = \boldsign_\mathbf{D}(\varphi(Q'))$ holds.
	\end{claim}
	\begin{claimproof}
		If $E(Q') \cap C = \varnothing$, the claim follows trivially.
		Suppose then that $E(Q') \cap C = \{e, f\}$.
		Let $w''$ be the vertex in $V''$ that is an end of $f$.
		Thus, for every $i$, $1 \leq i \leq \pf(G)$, we have that
			\[ |\fw_{D'_i}(Q')| = |\fw_{D_i}(\varphi(Q'))| - |\fw_{D_i}(P(w''))|. \]
		We conclude that $\sign_{D'_i}(Q') = \sign_{D_i}(\varphi(Q'))$ from~\eqref{eq:evenly-oriented-paths}.
	\end{claimproof}

	Let $\mathbf{v} := \boldsign_\mathbf{D'}(M') \odot \boldsign_\mathbf{D}(M)$ be a vector.
	We will show that, for every perfect matching $N'$ of $G'$, there is a perfect matching $N$ of $G$, such that 
	$	\boldsign_{\mathbf{D}'}(N')                  =
		\boldsign_{\mathbf{D}}(N)            \odot
		\mathbf{v}                  $.
	This implies that $\boldalpha \odot \mathbf{v}$ is a solution of $\mathbf{D'}$, where $\boldalpha$ is the solution of $\mathbf{D}$.

	Let $N'$ be a perfect matching of $G'$.
	Let $\mathcal{C}$ be the set of $(M', N')$-alternating cycles.
	By Proposition~\ref{prop:boldsign-product}, we know that
	\[ \boldsign_{\mathbf{D'}}(N') = \boldsign_\mathbf{D'}(M') \odot \prod_{Q' \in \mathcal{C}}\boldsign_\mathbf{D'}(Q'). \]
	From Claim~\ref{claim:equal-signs} and the definition of $\mathbf{v}$, we obtain that
	\[ \boldsign_{\mathbf{D'}}(N') = (\boldsign_\mathbf{D}(M) \odot \prod_{Q' \in \mathcal{C}}\boldsign_\mathbf{D}(\varphi(Q'))) \odot \mathbf{v}. \]
	Defining $N := M \symdiff (\bigcup_{Q' \in \mathcal{C}} E(\varphi(Q')))$, a perfect matching of $G$, we conclude that 
	\[	\boldsign_{\mathbf{D}'}(N')                  =
		\boldsign_{\mathbf{D}}(N)            \odot
		\mathbf{v}                  \] 
	as we desired.
	As mentioned above, this implies that $\mathbf{D'}$ has a solution and, therefore, is pfaffian.
	By a symmetric argument, we conclude that the contraction of $\mathbf{D}$ to $G''$ is also pfaffian.
\end{proof}

The theorem above implies a lower bound for the pfaffian number of a graph in terms of the pfaffian numbers of its separating cut contractions.

\begin{corollary}\label{cor:separating-cut-lower-bound}
	Let $G$ be a matching covered  graph and let $C = \partial(X)$
	be a separating cut of~$G$.  Then
		\[ \pf(G) \geq \max\{\pf(G/X), \pf(G/\overline{X})\}. \tag*{\qed} \]
\end{corollary}

Together    with Theorem~\ref{thm:pfaffian-tight-cut}  by Little  and Rendl and with Theorem~\ref{thm:norine-pfaffian-numbers} by Norine,  we obtain the following corollary.

\begin{corollary}
  Let $G$ be a matching covered graph  with $\pf(G) = 4$.  There is at
  least  one  graph  with  pfaffian   number  four  in  the  tight  cut
  decomposition of $G$.  \qed
\end{corollary}

One might believe that, as a generalization of Theorem~\ref{thm:pfaffian-tight-cut}, the pfaffian number of a graph is equal to the pfaffian number of one of its $C$-contractions if $C$ is tight.
This is not true.
As a counterexample, consider the graph $G_{19}$ in Figure~\ref{fig:g19} discovered by Miranda and Lucchesi~\cite{milu11}.
The cut $C$ shown in this figure is tight.
We know that $\pf(G_{19}) = 6$, but the $C$-contractions of this graph
are isomorphic to $K_{3,3}$, up to multiple edges, so each of them
have pfaffian number four.
However, equality holds in Theorem~\ref{thm:pfaffian-tight-cut} for a specific type of tight cut contractions, the bicontractions.

\begin{figure}[tb]
	\centering
	\includegraphics{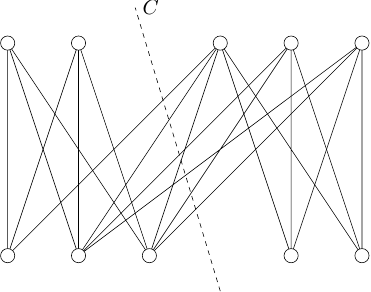}
	\caption{
		The graph $G_{19}$.
		\label{fig:g19}
	}
\end{figure}

Let $v$ be a vertex of degree two of a matching covered graph $G$ with four or more vertices.
Let $u$ and $w$ be the neighbours of $v$.
Then, the cut $\partial(X)$, where $X := \{u, v, w\}$, is tight and the graph $G/X$ is said to be a \emph{bicontraction} of $v$ from $G$.
The \emph{retract} of $G$, denoted by $\widehat{G}$, is a matching covered graph obtained from $G$ by repeated bicontractions whenever possible.
In particular, $G = \widehat{G}$ if there are no vertices of degree two.
It is easy to see that different orders of bicontractions in this process generate the same graph up to isomorphism;
therefore, the retract of a graph is well-defined.
Notice that $G$ is a bisubdivision of~$\widehat{G}$.

\begin{proposition}\label{prop:pf-retract}
	Let $G$ be a matching covered graph.
	Then
		\[ \pf(G) = \pf(\widehat{G}). \]
\end{proposition}
\begin{proof}
	If $G$ has no vertices of degree two, the statement follows trivially.
	Then, suppose $v$ is a vertex of degree two in $G$ with neighbours $u$ and $w$.
	Let $H$ be the graph obtained from $G$ by bicontraction of $v$.
	We will show that $\pf(G) = \pf(H)$.

	By Corollary~\ref{cor:separating-cut-lower-bound}, we know that $\pf(G) \geq \pf(H)$.
	Therefore, it suffices to prove that \mbox{$\pf(H) \geq \pf(G)$}.
	Let $\mathbf{D_H}$ be a $\pf(H)$-orientation of $H$.
	Extend $\mathbf{D_H}$ to a $\pf(H)$-orientation $\mathbf{D}$ of $G$ such that $(u, v, w)$ is an oriented path.
	Let $Q$ be a conformal cycle of $G$ that contains the vertex~$v$ and let $Q_H$ be the bicontraction of $v$ from $Q$.
	By definition of sign of cycles, we have that $\boldsign_\mathbf{D}(Q) = \boldsign_\mathbf{D_H}(Q_H)$.
	This implies that the sets of all signs of conformal cycles of $\mathbf{D_H}$ and of $\mathbf{D}$ are equal. 
	Let $M$ be a perfect matching of $G$ that contains $uv$ and let $M_H := M - uv$.
	Let $\mathbf{v} := \boldsign_\mathbf{D}(M) \odot \boldsign_\mathbf{D_H}(M_H)$ be a vector.
	It is easy to see that, for every perfect matching $N$ of $G$, there is a perfect matching $N_H$ of $H$, such that 
	\[	\boldsign_{\mathbf{D}}(N)                  =
		\boldsign_{\mathbf{D_H}}(N_H)            \odot
		\mathbf{v}. \]
	This implies that $\mathbf{D}$ is pfaffian.
	The proposition follows by induction.
\end{proof}

\section{Conclusions}
\label{sec:conclusion}

We proved lower bounds for the pfaffian number of graphs.
Lower bounds for a closely related method, namely the symbolic pfaffian method, were already known.
We established a relation between the pfaffian number and the symbolic pfaffian number: for a graph $G$, the relation $\pf(G) \geq \spf(G) + 1$ holds.
However, the condition for the existence of a pfaffian $k$-orientation is more complicated than the condition for the existence of a pfaffian $k$-symbolic orientation.
In that sense, we believe that the growth rate of the pfaffian number might be much faster than the growth rate of the symbolic pfaffian number.
In fact, the growth rate of the pfaffian number might be exponential for complete graphs.
In that regard, we conjecture that the lower bound of Corollary~\ref{cor:pf-conformal-lower-bound} is not sharp if both conformal subgraphs are not pfaffian.
\begin{conjecture}
	Let $G$ be a matchable graph.
	Let $H'$ be a non-spanning conformal subgraph of $G$ and let $H'' := G - V(H')$.
	If $H'$ and $H''$ are not pfaffian, then
		\[ \pf(H') + \pf(H'') - 1 < \pf(G).  \]
\end{conjecture}

We also proved that the pfaffian number of a graph does not
increase after a separating cut contraction.
In particular, a bicontraction preserves the pfaffian number.
It is an interesting open problem to characterize separating cuts
whose contractions  preserve the pfaffian number.

\bibliography{paper}

\begin{thebibliography}{10}

\bibitem{bavy17}
Alina~V. Babenko and Mikhail~N. Vyalyi.
\newblock On the linear classification of even and odd permutation matrices and
  the complexity of computing the permanent.
\newblock {\em Computational Mathematics and Mathematical Physics},
  57(2):362–371, Feb 2017.
\newblock \href {https://doi.org/10.1134/s0965542517020038}
  {\path{doi:10.1134/s0965542517020038}}.

\bibitem{brush1967}
Stephen~G. Brush.
\newblock History of the {L}enz-{I}sing model.
\newblock {\em Reviews of Modern Physics}, 39(4):883, 1967.

\bibitem{clm12}
Marcelo~Henriques Carvalho, Cl\'audio~L. Lucchesi, and Uppaluri S.~R. Murty.
\newblock A generalization of {L}ittle's theorem on {P}faffian orientations.
\newblock {\em Journal of Combinatorial Theory, Series B}, 102:1241--1266,
  2012.

\bibitem{cms21}
Roberta R.~M. {Costa Moço}, Alberto A.~A. Miranda, and C\^andida~Nunes
  da~Silva.
\newblock The signature matrix for 6-{Pfaffian} graphs.
\newblock {\em Procedia Computer Science}, 195:298--305, 2021.
\newblock Proceedings of the XI Latin and American Algorithms, Graphs and
  Optimization Symposium.

\bibitem{galo99}
Anna Galluccio and Martin Loebl.
\newblock On the theory of {P}faffian orientations. {I. P}erfect matchings and
  permanents.
\newblock {\em Electronic Journal of Combinatorics}, 6, 1999.

\bibitem{gutman1977}
Ivan Gutman.
\newblock Topological properties of benzenoid systems.
\newblock {\em Theoretica Chimica Acta}, 45:309--315, 1977.

\bibitem{herndon1974}
William~C. Herndon.
\newblock Resonance theory and the enumeration of {K}ekule structures.
\newblock {\em Journal of Chemical Education}, 51(1):10, 1974.

\bibitem{kast61}
Pieter~W. Kasteleyn.
\newblock The statistics of dimers on a lattice: {I. T}he number of dimer
  arrangements on a quadratic lattice.
\newblock {\em Physica}, 27(12):1209--1225, 1961.

\bibitem{kast63}
Pieter~W. Kasteleyn.
\newblock Dimer statistics and phase transitions.
\newblock {\em Journal of Mathematical Physics}, 4:287--293, 1963.

\bibitem{kast1967}
Pieter~W. Kasteleyn.
\newblock Graph theory and crystal physics.
\newblock {\em Graph theory and theoretical physics}, pages 43--110, 1967.

\bibitem{kgz2017}
Mario Krenn, Xuemei Gu, and Anton Zeilinger.
\newblock Quantum {E}xperiments and {G}raphs: Multiparty {S}tates as {C}oherent
  {S}uperpositions of {P}erfect {M}atchings.
\newblock {\em Physical review letters}, 119(24):240403, 2017.

\bibitem{lire91}
Charles H.~C. Little and Franz Rendl.
\newblock Operations preserving the {P}faffian property of a graph.
\newblock {\em Journal of the Australian Mathematical Society (Series A)},
  50:248--275, 1991.

\bibitem{lova87}
L\'aszl\'o Lov\'asz.
\newblock Matching structure and the matching lattice.
\newblock {\em Journal of Combinatorial Theory, Series B}, 43:187--222, 1987.

\bibitem{lopl86}
L\'aszl\'o Lov\'asz and Michael~D. Plummer.
\newblock {\em Matching Theory}.
\newblock Number~29 in Annals of Discrete Mathematics. Elsevier Science, 1986.

\bibitem{lm2024}
Cl\'audio~L. Lucchesi and Uppaluri S.~R. Murty.
\newblock {\em Perfect Matchings: {}A Theory of Matching Covered Graphs}.
\newblock Springer International Publishing AG, 2024.

\bibitem{msv99}
Meena Mahajan, P.~R. Subramanya, and V.~Vinay.
\newblock A combinatorial algorithm for pfaffians.
\newblock In Takano Asano, Hideki Imai, D.~T. Lee, Shin-ichi Nakano, and
  Takeshi Tokuyama, editors, {\em Computing and Combinatorics}, pages 134--143,
  Berlin, Heidelberg, 1999. Springer Berlin Heidelberg.

\bibitem{milu11}
Alberto A.~A. Miranda and Cl\'audio~L. Lucchesi.
\newblock Matching signatures and pfaffian graphs.
\newblock {\em Discrete Mathematics}, 311:289--294, 2011.

\bibitem{nori09}
Serguei Norine.
\newblock Drawing 4-{P}faffian graphs on the torus.
\newblock {\em Combinatorica}, 29:109--119, 2009.

\bibitem{tesl00}
Glenn Tesler.
\newblock Matchings in graphs on non-orientable surfaces.
\newblock {\em Journal of Combinatorial Theory, Series B}, 78:189--231, 2000.

\bibitem{vali79}
Leslie Valiant.
\newblock The complexity of computing the permanent.
\newblock {\em Theoretical Computer Science}, 8:189--201, 1979.

\bibitem{vyalyi21}
Mikhail~N. Vyalyi.
\newblock Counting the number of perfect matchings, and generalized decision
  trees.
\newblock {\em Problems of Information Transmission}, 57(2):143–160, Apr
  2021.
\newblock \href {https://doi.org/10.1134/s0032946021020046}
  {\path{doi:10.1134/s0032946021020046}}.

\end{thebibliography}

\end{document}